\documentclass[reqno,11pt]{article}
\usepackage{a4wide,graphicx,bm,amssymb,ifthen,url}
\usepackage{color,epsf,cancel}
\usepackage[arrow, matrix, curve]{xy}
\usepackage{color}
\usepackage{colortbl}
\usepackage{lmodern}
\usepackage{MnSymbol}
\usepackage{amsthm}
\usepackage{amsmath,amssymb,amstext}
\usepackage{graphicx}
\usepackage{tikz}
\usepackage{nicefrac}

\newtheorem{theorem}{Theorem}
\newtheorem{lemma}{Lemma}
\newtheorem{property}{Property}
\newtheorem{proposition}{Proposition}
\newtheorem{conjecture}{Conjecture}
\newtheorem{corollary}{Corollary}

\newtheorem{remark}{Remark}

\definecolor{magenta}{rgb}{0.75,0,0.25}

\definecolor{violet}{rgb}{0.25,0,0.75}

\definecolor{dgreen}{rgb}{0.0,0.5,0.0}

\newcommand{\A}{\includegraphics{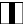}}
\newcommand{\B}{\includegraphics{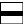}}
\newcommand{\C}{\includegraphics{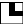}}
\newcommand{\D}{\includegraphics{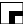}}
\newcommand{\E}{\includegraphics{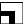}}
\newcommand{\F}{\includegraphics{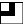}}

\newcommand{\Aa}{{\mathcal A}}

\newcommand {\W}  {{\mathcal W}}

\newcommand {\V}  {{\mathcal V}}
\newcommand {\G}  {{\mathcal G}}

\newcommand{\linf}{L_{\infty}}

\title{On the dimension of arcs in mixed labyrinth fractals}

\author{
Ligia L. Cristea \thanks{This author's work was supported by the Austrian Science Fund (FWF), 
Projects P27050-N26, P29910-N35 and F-5512.} 
\\Technische Universit\"at Graz, Institut f\"ur Mathematik\\ Kopernikusgasse 24,
 8010 Graz, Austria\\ \tt{strublistea@gmail.com} 
\\ \and Damir Vuki\v{c}evi\'c 
\\ Department of Mathematics, Faculty of Science,\\ University of Split, Rugjera Boskovica 33, HR-21000 Split, Croatia,\\ \tt{damir.vukicevic@pmfst.hr}
}

\date{\today}

\begin{document}

\maketitle


\abstract{ Mixed labyrinth fractals are dendrites in the unit square, which were studied recently with respect to the lengths of arcs in the fractals. In this article we first give a construction method for mixed labyrinth fractals with the property that all arcs in the fractal have box counting dimension $2$. Subsequently, we show how a certain familiy of patterns can provide a mixed labyrinth fractal of any box counting dimension between $1$ and $2$, which also coincides with the box-counting dimension of the arc between any two distinct points of the fractal. Finally, we show how the results can be extended to a more general setting.}

\section{Introduction}

Labyrinth fractals are dendrites in the unit square which were introduced in the last decade \cite{laby_4x4,laby_mxm}. They can also be viewed as special families of Sierpi\'nski carpets. Labyrinth fractals  were studied with respect to their  topological and geometrical  properties, with special focus on the arcs in these dendrites. The construction of labyrinth sets and labyrinth fractals is done  by using labyrinth patterns: in the self-similar case one labyrinth pattern generates, inductively, a sequence of labyrinth sets, whose limit is a labyrinth fractal, and in the case of mixed labyrinth sets and mixed labyrinth fractals  \cite{mixlaby} more than one pattern (in general, even infinitely many) are used for their construction.

While for the self-similar labyrinth fractals \cite{laby_4x4,laby_mxm} there are results on the lengths and the dimension of the arcs in these dendrites, the papers on the mixed case \cite{mixlaby,arcs_mixlaby} deal mainly with the length of arcs, and how this depends on the properties of the patterns that generate a mixed labyrinth fractal. The main goal of the present article is to provide examples for the construction of families of mixed labyrinth fractals with the property that the arcs between distinct points of the fractal have given box counting dimension, equal to the dimension of the whole fractal. 

There exists quite a lot of interest in fractal dendrites lately. Self-similar dendrites were constructed with the help of polygonal systems and studied recently \cite{SamuelTetenov_selfsimilar_dendrites}, but the emphasis is not on length of arcs, and, moreover, the methods are different, using  tools like contractive iterated function systems and zippers, a different approach from ours, and restricted only to self-similar fractals.
We also mention existing results on dimensions of homogenoeous Moran sets in a closed real interval of finite length \cite{FengWenWu_dim_homogeneous_Moran_sets1997}, since the fractals studied in the present paper can be seen as a family of planar relatives of these.

In the last years,  labyrinth fractals and results on them
have already shown to have applications in physics. We mention, e.g., 
the fractal reconstruction of complicated
images, signals and radar backgrounds \cite{PotapovGermanGrachev2013}, or 
the construction of prototypes of ultra-wide band radar antennas
\cite{PotapovZhang2016} . Very recently,
 fractal labyrinths combined 
with genetic algorithms were used in order to create big robust antenna arrays and
nano-antennas in  telecommunication \cite{PotapovPotapovPotapov_dec2017}. 
Concerning further possible  applications of mixed labyrinth fractals, 
we want to mention very recent results in materials engineering \cite{JanaGarcia_lithiumdendrite2017}, which show that dendrite growth, a largely unsolved problem,  plays an essential role when  approaching high power and energy lithium-ion batteries. Furthermore, there is recent research on crystal growth \cite{tarafdar_multifractalNaCl2013} which suggests that, depending on the geometric structure of the studied material, labyrinth fractals can offer a suitable model for studying various phenomena and objects occurring in other fields of science.
Finally, we mention that Koch curves, which are of interest in
theoretical physics in the context of diffusion processes ({e.g., }
\cite{SeegerHoffmannEssex2009_randomKoch}), are related to arcs  in  labyrinth fractals.

Let us now give a short outline of the article. 
  Section \ref{sec:preliminaries} contains the needed preliminaries: we give the definition of labyrinth patters, labyrinth sets and labyrinth fractals, and cite some of the main results regarding topological and geometric properties of labyrinth fractals proven in earlier research \cite{laby_4x4,laby_mxm,mixlaby,supermix}, which are essential to our considerations throughout this paper.
In Section \ref{sec:snake_cross_patterns} we introduce a special family of horizontally and vertically blocked labyrinth patterns called snake cross patterns. We analyse their shape, as well as the the objects  generated by certain sequences of snake cross patterns. These are mixed labyrinth sets and, in the limit, mixed labyrinth fractals, dendrites with the property that the length of any non-trivial arc is infinite.
 Section \ref{sec:mixed_arcs_max_dim}     is dedicated to proving  that the fractal dendrite obtained as the limit of the construction in the previous section has the property that the arc between any two distinct points in the fractal has box counting dimension two. Subsequently, we give in Section \ref{sec:omnidimensional} a construction based on the use of snake cross patterns that provides mixed labyrinth fractals of any box-counting dimension $\delta \in [1,2)$. In a short last section we take a glipmse at how these results could be used to infere analogons thereof in a more general setting. We conclude with a conjecture.

\section{Preliminaries on self-similar and mixed labyrinth fractals}\label{sec:preliminaries}

 \emph{Labyrinth patterns} are very convenient tools for the construction labyrinth fractals. The notations and notations have already been introduced and used in work on self-similar labyrinth fractals \cite{laby_4x4, laby_mxm} and later on for the mixed ones \cite{mixlaby}.
Let $x,y,q\in [0,1]$ such that $Q=[x,x+q]\times [y,y+q]\subseteq [0,1]\times [0,1]$. 
For any point $(z_x,z_y)\in[0,1]\times [0,1]$ we define the map
%\[
$P_Q(z_x,z_y)=(q z_x+x,q z_y+y)$.
%\]
For any integer $m\ge 1$ let $S_{i,j,m}=\{(x,y)\mid \frac{i}{m}\le x \le \frac{i+1}{m} \mbox{ and } \frac{j}{m}\le y \le \frac{y+1}{m} \}$ and  
${\cal S}_m=\{S_{i,j,m}\mid 0\le i\le m-1 \mbox{ and } 0\le j\le m-1 \}$. 

Any nonempty ${\cal A} \subseteq {\cal S}_m$ is called an $m$-\emph{pattern} and $m$ its \emph{width}. Let $\{{\cal A}_k\}_{k=1}^{\infty}$ 
be a sequence of non-empty patterns and $\{m_k\}_{k=1}^{\infty}$ be the corresponding 
\emph{width-sequence}, i.e., for all $k\ge 1$ we have 
${\cal A}_k\subseteq {\cal S}_{m_k}$. 
We denote $m(n)=\prod_{k=1}^n m_k$, for all $n \ge 1$. 
Let ${\cal W}_1={\cal A}_{1}$. Ee call ${\cal W}_1$ the 
\emph{set of white squares of level $1$}, and 
the elements of ${\cal S}_{m_1} \setminus {\cal W}_1$ the \emph{black squares of level $1$}.
 For $n\ge 2$ the \emph{set of white squares of level $n$} is defined as
\begin{equation} \label{eq:W_n}
{\cal W}_n=\bigcup_{W\in {\cal A}_{n}, W_{n-1}\in {\cal W}_{n-1}}\{ P_{W_{n-1}}(W)\}.
\end{equation}

 \noindent Note that ${\cal W}_n\subset {\cal S}_{m(n)}$. 
We call the elements of${\cal S}_{m(n)} \setminus {\cal W}_n$ the \emph{black squares of level} $n$. For $n\ge 1$, we define $L_n=\bigcup_{W\in {\cal W}_n} W$. 
Thus, $\{L_n\}_{n=1}^{\infty}$ is a monotonically decreasing sequence of compact sets, and $L_{\infty}=\bigcap_{n=1}^{\infty}L_n$,  is the \emph{limit set defined by the sequence of patterns 
$\{{\cal A}_k\}_{k=1}^{\infty}.$ }

Examples of patterns are shown in figures~\ref{fig:A1A2} and \ref{fig:W2}.  These also illustrate the first two steps of the 
construction of a mixed labyrinth set. 

We define, for ${\cal A}\subseteq {{\cal S}_m}$, \emph{the graph of the pattern ${\cal A}$}, $\mathcal{G}({\cal A})\equiv (\mathcal{V}(\mathcal{G}({\cal A})),\mathcal{E}(\mathcal{G}({\cal A})))$: Its set of vertices $\mathcal{V}(\mathcal{G}({\cal A}))$ consists of the white squares in ${\cal A}$, i.e.,  $\mathcal{V}(\mathcal{G}({\cal A}))={\cal A}$ and its edges $\mathcal{E}(\mathcal{G}({\cal A}))$ are the unordered pairs of white squares, that have a common side. The \emph{top row} in ${\cal A}$ is the set of all white squares in $\{S_{i,m-1,m}\mid 0\le i\le m^n-1 \}$. The bottom row, left column, and right column in ${\cal A}$  are defined analogously. A \emph{top exit} in ${\cal A}$ is a white square in the top row, such that there is a white square in the same column in the bottom row. A \emph{bottom exit} in ${\cal A}$  is defined analogously. A \emph{left exit} in ${\cal A}$ is a white square in the left column, such that there is a white square in the same row in the right column. A \emph{right exit} in ${\cal A}$ is defined analogously. A top exit together with the corresponding bottom exit build a \emph{vertical exit pair},  and a left exit together with the corresponding right exit build a \emph{horizontal exit pair}.
The above notions can also be defined in the special case ${\cal A}={\cal W}_{n}$. 

%---------------neu-------------------
\begin{figure}[hhhh]
\begin{center}
\begin{tikzpicture}[scale=.36]
\draw[line width=1pt] (0,0) rectangle (10,10);
\draw[line width=1pt] (2.5, 0) -- (2.5,10);
\draw[line width=1pt] (5, 0) -- (5,10);
\draw[line width=1pt] (7.5, 0) -- (7.5,10);
\draw[line width=1pt] (0, 2.5) -- (10,2.5);
\draw[line width=1pt] (0, 5) -- (10,5);
\draw[line width=1pt] (0, 7.5) -- (10,7.5);
%\fill[color=black]  (2.7, 0) rectangle (2.5, 2.5);
%\fill[color=black]  (2.5, 0) rectangle (2.5, 2.5);
% \fill[color=black]  (5, 0) rectangle (2.5, 2.5);
% \fill[color=black]  (7.5, 0) rectangle (2.5, 2.5);
% \fill[color=black]  (5, 2.5) rectangle (2.5,2.5);
% \fill[color=black]  (0, 5) rectangle (2.5,2.5);
% \fill[color=black]  (5, 7.5 ) rectangle (2.5,2.5);
% \fill[color=black]  (7.5, 7.5) rectangle (2.5,2.5);
%\draw[fill=black] (2.5, 0) rectangle (2.5,2.5);
\filldraw[fill=black, draw=black] (2.5,0) rectangle (5, 2.5);
\filldraw[fill=black, draw=black] (5,0) rectangle (7.5, 2.5);
\filldraw[fill=black, draw=black] (7.5,0) rectangle (10, 2.5);
\filldraw[fill=black, draw=black] (5,2.5) rectangle (7.5, 5);
\filldraw[fill=black, draw=black] (0,5) rectangle (2.5, 7.5);
\filldraw[fill=black, draw=black] (5,7.5) rectangle (7.5, 10);
\filldraw[fill=black, draw=black] (7.5,7.5) rectangle (10, 10);
\draw[line width=1pt] (11,0) rectangle (21,10);
\draw[line width=1pt] (13, 0) -- (13,10);
\draw[line width=1pt] (15, 0) -- (15,10);
\draw[line width=1pt] (17, 0) -- (17,10);
\draw[line width=1pt] (19, 0) -- (19,10);
\draw[line width=1pt] (11, 2) -- (21,2);
\draw[line width=1pt] (11, 4) -- (21,4);
\draw[line width=1pt] (11, 6) -- (21,6);
\draw[line width=1pt] (11, 8) -- (21,8);
\filldraw[fill=black, draw=black] (13,0) rectangle (15, 2);
\filldraw[fill=black, draw=black] (17,0) rectangle (19, 2);
\filldraw[fill=black, draw=black] (19,0) rectangle (21, 2);
\filldraw[fill=black, draw=black] (19,2) rectangle (21, 4);
\filldraw[fill=black, draw=black] (11,4) rectangle (13, 6);
\filldraw[fill=black, draw=black] (15,4) rectangle (17, 6);
\filldraw[fill=black, draw=black] (15,6) rectangle (17, 8);
\filldraw[fill=black, draw=black] (17,6) rectangle (19, 8);
\filldraw[fill=black, draw=black] (11,8) rectangle (13, 10);
\filldraw[fill=black, draw=black] (17,8) rectangle (19, 10 );
\filldraw[fill=black, draw=black] (19,8) rectangle (21, 10);
\end{tikzpicture}
\caption{Two labyrinth patterns, ${\cal A}_1$ (a $4$-pattern) and ${\cal A}_2$ (a $5$-pattern)}\label{fig:A1A2}
\end{center}
\end{figure}
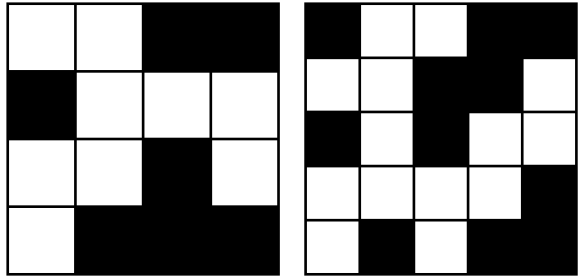

%-------------------------------------------------------------------
\begin{figure}[hhhh]\label{W2}
\begin{center}
\begin{tikzpicture}[scale=.18]
\draw[line width=1pt] (0,0) rectangle (20,20);
%------------------------grid of level 1
\draw[line width=0.8pt] (5, 0) -- (5,20);
\draw[line width=0.8pt] (10, 0) -- (10,20);
\draw[line width=0.8pt] (15, 0) -- (15,20);
\draw[line width=0.8pt] (0, 5) -- (20,5);
\draw[line width=0.8pt] (0, 10) -- (20,10);
\draw[line width=0.8pt] (0, 15) -- (20,15);
%---------------------------grid of level 2
%----- vertical
\draw[line width=0.5pt] (1, 0) -- (1,20);
\draw[line width=0.5pt] (2, 0) -- (2,20);
\draw[line width=0.5pt] (3, 0) -- (3,20);
\draw[line width=0.5pt] (4, 0) -- (4,20);
\draw[line width=0.5pt] (6, 0) -- (6,20);
\draw[line width=0.5pt] (7, 0) -- (7,20);
\draw[line width=0.5pt] (8, 0) -- (8,20);
\draw[line width=0.5pt] (9, 0) -- (9,20);
\draw[line width=0.5pt] (11, 0) -- (11,20);
\draw[line width=0.5pt] (12, 0) -- (12,20);
\draw[line width=0.5pt] (13, 0) -- (13,20);
\draw[line width=0.5pt] (14, 0) -- (14,20);
\draw[line width=0.5pt] (16, 0) -- (16,20);
\draw[line width=0.5pt] (17, 0) -- (17,20);
\draw[line width=0.5pt] (18, 0) -- (18,20);
\draw[line width=0.5pt] (19, 0) -- (19,20);
%-------------------horizontal
\draw[line width=0.5pt] (0, 1) -- (20,1);
\draw[line width=0.5pt] (0, 2) -- (20,2);
\draw[line width=0.5pt] (0, 3) -- (20,3);
\draw[line width=0.5pt] (0, 4) -- (20,4);
\draw[line width=0.5pt] (0, 6) -- (20,6);
\draw[line width=0.5pt] (0, 7) -- (20,7);
\draw[line width=0.5pt] (0, 8) -- (20,8);
\draw[line width=0.5pt] (0, 9) -- (20,9);
\draw[line width=0.5pt] (0, 11) -- (20,11);
\draw[line width=0.5pt] (0, 12) -- (20,12);
\draw[line width=0.5pt] (0, 13) -- (20,13);
\draw[line width=0.5pt] (0, 14) -- (20,14);
\draw[line width=0.5pt] (0, 16) -- (20,16);
\draw[line width=0.5pt] (0, 17) -- (20,17);
\draw[line width=0.5pt] (0, 18) -- (20,18);
\draw[line width=0.5pt] (0, 19) -- (20,19);
%--- black quares of level 1
\filldraw[fill=black, draw=black] (5,0) rectangle (10, 5);
\filldraw[fill=black, draw=black] (10,0) rectangle (15, 5);
\filldraw[fill=black, draw=black] (15,0) rectangle (20, 5);
\filldraw[fill=black, draw=black] (10,5) rectangle (15, 10);
\filldraw[fill=black, draw=black] (0,10) rectangle (5, 15);
\filldraw[fill=black, draw=black] (10,15) rectangle (15, 20);
\filldraw[fill=black, draw=black] (15,15) rectangle (20, 20);
%----- black squares of level 2
\filldraw[fill=black, draw=black] (1,0) rectangle (2, 1);
\filldraw[fill=black, draw=black] (3,0) rectangle (4, 1);
\filldraw[fill=black, draw=black] (4,0) rectangle (5, 1);
\filldraw[fill=black, draw=black] (4,1) rectangle (5, 2);
\filldraw[fill=black, draw=black] (0,2) rectangle (1, 3);
\filldraw[fill=black, draw=black] (2,2) rectangle (3, 3);
\filldraw[fill=black, draw=black] (2,3) rectangle (3, 4);
\filldraw[fill=black, draw=black] (3,3) rectangle (4, 4);
\filldraw[fill=black, draw=black] (0,4) rectangle (1, 5);
\filldraw[fill=black, draw=black] (3,4) rectangle (4, 5 );
\filldraw[fill=black, draw=black] (4,4) rectangle (5, 5);
%-----------+(0,5)----------------------------------
\filldraw[fill=black, draw=black] (1,5) rectangle (2, 6);
\filldraw[fill=black, draw=black] (3,5) rectangle (4, 6);
\filldraw[fill=black, draw=black] (4,5) rectangle (5, 6);
\filldraw[fill=black, draw=black] (4,6) rectangle (5, 7);
\filldraw[fill=black, draw=black] (0,7) rectangle (1, 8);
\filldraw[fill=black, draw=black] (2,7) rectangle (3, 8);
\filldraw[fill=black, draw=black] (2,8) rectangle (3, 9);
\filldraw[fill=black, draw=black] (3,8) rectangle (4, 9);
\filldraw[fill=black, draw=black] (0,9) rectangle (1, 10);
\filldraw[fill=black, draw=black] (3,9) rectangle (4, 10);
\filldraw[fill=black, draw=black] (4,9) rectangle (5, 10);
%----------+(5,5)----------------------------------------
\filldraw[fill=black, draw=black] (6,5) rectangle (7, 6);
\filldraw[fill=black, draw=black] (8,5) rectangle (9, 6);
\filldraw[fill=black, draw=black] (9,5) rectangle (10, 6);
\filldraw[fill=black, draw=black] (9,6) rectangle (10, 7);
\filldraw[fill=black, draw=black] (5,7) rectangle (6, 8);
\filldraw[fill=black, draw=black] (7,7) rectangle (8, 8);
\filldraw[fill=black, draw=black] (7,8) rectangle (8, 9);
\filldraw[fill=black, draw=black] (8,8) rectangle (9, 9);
\filldraw[fill=black, draw=black] (5,9) rectangle (6, 10);
\filldraw[fill=black, draw=black] (8,9) rectangle (9, 10);
\filldraw[fill=black, draw=black] (9,9) rectangle (10, 10);
%---- --+(5,10)------------------------------------------
\filldraw[fill=black, draw=black] (6,10) rectangle (7, 11);
\filldraw[fill=black, draw=black] (8,10) rectangle (9, 11);
\filldraw[fill=black, draw=black] (9,10) rectangle (10, 11);
\filldraw[fill=black, draw=black] (9,11) rectangle (10, 12);
\filldraw[fill=black, draw=black] (7,12) rectangle (6, 13);
\filldraw[fill=black, draw=black] (7,12) rectangle (8, 13);
\filldraw[fill=black, draw=black] (7,13) rectangle (8, 14);
\filldraw[fill=black, draw=black] (8,13) rectangle (9, 14);
\filldraw[fill=black, draw=black] (5,14) rectangle (6, 15);
\filldraw[fill=black, draw=black] (8,14) rectangle (9, 15 );
\filldraw[fill=black, draw=black] (9,14) rectangle (10, 15);
%----------+(10,10)-----------------------------------------
\filldraw[fill=black, draw=black] (11,10) rectangle (12, 11);
\filldraw[fill=black, draw=black] (13,10) rectangle (14, 11);
\filldraw[fill=black, draw=black] (14,10) rectangle (15, 11);
\filldraw[fill=black, draw=black] (14,11) rectangle (15, 12);
\filldraw[fill=black, draw=black] (10,12) rectangle (11, 13);
\filldraw[fill=black, draw=black] (12,12) rectangle (13, 13);
\filldraw[fill=black, draw=black] (12,13) rectangle (13, 14);
\filldraw[fill=black, draw=black] (13,13) rectangle (14, 14);
\filldraw[fill=black, draw=black] (10,14) rectangle (11, 15);
\filldraw[fill=black, draw=black] (13,14) rectangle (14, 15 );
\filldraw[fill=black, draw=black] (14,14) rectangle (15, 15);
%----------+(15,5)------------------
\filldraw[fill=black, draw=black] (16,5) rectangle (17, 6);
\filldraw[fill=black, draw=black] (18,5) rectangle (19, 6);
\filldraw[fill=black, draw=black] (19,5) rectangle (20, 6);
\filldraw[fill=black, draw=black] (19,6) rectangle (20, 7);
\filldraw[fill=black, draw=black] (15,7) rectangle (16, 8);
\filldraw[fill=black, draw=black] (17,7) rectangle (18, 8);
\filldraw[fill=black, draw=black] (17,8) rectangle (18, 9);
\filldraw[fill=black, draw=black] (18,8) rectangle (19, 9);
\filldraw[fill=black, draw=black] (15,9) rectangle (16, 10);
\filldraw[fill=black, draw=black] (18,9) rectangle (19, 10 );
\filldraw[fill=black, draw=black] (19,9) rectangle (20, 10);
%---------+(15,10)-----------------
\filldraw[fill=black, draw=black] (16,10) rectangle (17, 11);
\filldraw[fill=black, draw=black] (18,10) rectangle (19, 11);
\filldraw[fill=black, draw=black] (19,10) rectangle (20, 11);
\filldraw[fill=black, draw=black] (19,11) rectangle (20, 12);
\filldraw[fill=black, draw=black] (15,12) rectangle (16, 13);
\filldraw[fill=black, draw=black] (17,12) rectangle (18, 13);
\filldraw[fill=black, draw=black] (17,13) rectangle (18, 14);
\filldraw[fill=black, draw=black] (18,13) rectangle (19, 14);
\filldraw[fill=black, draw=black] (15,14) rectangle (16, 15);
\filldraw[fill=black, draw=black] (18,14) rectangle (19, 15);
\filldraw[fill=black, draw=black] (19,14) rectangle (20, 15);
%-----------+(0,15)------------
\filldraw[fill=black, draw=black] (1,15) rectangle (2, 16);
\filldraw[fill=black, draw=black] (3,15) rectangle (4, 16);
\filldraw[fill=black, draw=black] (4,15) rectangle (5, 16);
\filldraw[fill=black, draw=black] (4,16) rectangle (5, 17);
\filldraw[fill=black, draw=black] (0,17) rectangle (1, 18);
\filldraw[fill=black, draw=black] (2,17) rectangle (3, 18);
\filldraw[fill=black, draw=black] (2,18) rectangle (3, 19);
\filldraw[fill=black, draw=black] (3,18) rectangle (4, 19);
\filldraw[fill=black, draw=black] (0,19) rectangle (1, 20);
\filldraw[fill=black, draw=black] (3,19) rectangle (4, 20);
\filldraw[fill=black, draw=black] (4,19) rectangle (5, 20);
%----------+(5,15)--------------------
\filldraw[fill=black, draw=black] (6,15) rectangle (7, 16);
\filldraw[fill=black, draw=black] (8,15) rectangle (9, 16);
\filldraw[fill=black, draw=black] (9,15) rectangle (10, 16);
\filldraw[fill=black, draw=black] (9,16) rectangle (10, 17);
\filldraw[fill=black, draw=black] (5,17) rectangle (6, 18);
\filldraw[fill=black, draw=black] (7,17) rectangle (8, 18);
\filldraw[fill=black, draw=black] (7,18) rectangle (8, 19);
\filldraw[fill=black, draw=black] (8,18) rectangle (9, 19);
\filldraw[fill=black, draw=black] (5,19) rectangle (6, 20);
\filldraw[fill=black, draw=black] (8,19) rectangle (9, 20);
\filldraw[fill=black, draw=black] (9,19) rectangle (10, 20);
%--------+(15,15)-----------------
\filldraw[fill=black, draw=black] (16,15) rectangle (17, 16);
\filldraw[fill=black, draw=black] (18,15) rectangle (19, 16);
\filldraw[fill=black, draw=black] (19,15) rectangle (20, 16);
\filldraw[fill=black, draw=black] (19,16) rectangle (20, 17);
\filldraw[fill=black, draw=black] (15,17) rectangle (16, 18);
\filldraw[fill=black, draw=black] (17,17) rectangle (18, 18);
\filldraw[fill=black, draw=black] (17,18) rectangle (18, 19);
\filldraw[fill=black, draw=black] (18,18) rectangle (19, 19);
\filldraw[fill=black, draw=black] (15,19) rectangle (16, 20);
\filldraw[fill=black, draw=black] (18,19) rectangle (19, 20);
\filldraw[fill=black, draw=black] (19,19) rectangle (20, 20);
\end{tikzpicture}
\caption{The set ${\cal W}_2$, constructed based on the above patterns ${\cal A}_1$ and ${\cal A}_2$, that
can also be viewed as a $20$-pattern} \label{fig:W2}
\end{center}
\end{figure}
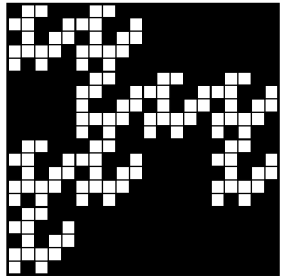
%------------------------------------------------------------------------

%------------------------------------
A non-empty $m$-pattern ${\cal A} \subseteq {{\cal S}_m}$, $m \ge 3$ is called a $m\times m$-\emph{labyrinth pattern} (in short, \emph{labyrinth pattern}) if  ${\cal A}$ has the tree property, the exits property and the corner property. 
%----------------------------------------------  
\begin{property}[The tree property]\label{prop1}
$\mathcal{G}({\cal A})$ is a tree.
\end{property}
%------------------------------------------------------------------------------------------------------------------------------------
\begin{property}[The exits property]\label{prop2}
${\cal A}$ has exactly one vertical exit pair, and exactly one horizontal exit pair.
\end{property}
%------------------------------------------------------------------------------------------------------------------------------------
\begin{property}[The corner property]\label{prop3}
If there is a white square in ${\cal A}$ at a corner of ${\cal A}$, then there is no white square in ${\cal A}$ at the diagonal opposite corner of ${\cal A}$. 
\end{property}

Let $\{{\cal A}_k\}_{k=1}^{\infty}$ be a sequence of non-empty patterns,  with $m_k\ge 3$, $n\ge 1$ and ${\cal W}_n$ the corresponding set of white squares of level $n$. We call ${\cal W}_{n}$ an $m(n)\times m(n)$-\emph{mixed labyrinth set} (in short, \emph{labyrinth set}) of level $n$, if ${\cal A} ={\cal W}_{n}$ 
has the tree property, the exits property and the corner property.
%---------------
It was shown \cite[Lemma 1]{mixlaby} that if all patterns in the sequence $\{{\cal A}_k\}_{k=1}^{\infty}$ are labyrinth patterns, then ${\cal W}_n$ is a labyrinth set, for any $n\ge 1$.
The limit set $L_{\infty}$ defined by a sequence $\{{\cal A}_k\}_{k=1}^{\infty}$ of labyrinth patterns is called \emph{mixed labyrinth fractal}. 

%--------------------------------------------------------------------------------------------------------------------------------------
The fact that the labyrinth patterns have exits leads to the fact that any labyrinth fractal also has $4$ exits, which are defined as follows. Let $T_n\in {\cal W}_{n}$ be the top exit of ${\cal W}_{n}$, for $n\ge 1$. The \emph{top exit of} 
$L_{\infty}$ is $\bigcap_{n=1}^{\infty}T_n$. The other exits of $L_{\infty}$ are defined analogously. We note that the exits property
yields that $(x,1),(x,0)\in L_{\infty}$ if and only if $(x,1)$ is the top exit of $L_{\infty}$ and $(x,0)$ 
is the bottom exit of $L_{\infty}$. For the left and the right exits the analogue statement holds.

For more details on mixed labyrinth sets and mixed labyrinth fractals and for results on topological properties of mixed labyrinth fractals we refer to the papers \cite{mixlaby,arcs_mixlaby}. 

In the special case when in the above construction $\Aa_k= \Aa$, for all $k\ge 1$, i.e., the sequence of labyrinth patterns is a constant sequence, we obtain a self-similar labyrinth fractal, like those studied initially \cite{laby_4x4,laby_mxm}.

The following result was proven both in the case of self-similar labyrinth fractals \cite[Theorem 1]{laby_4x4}, and in the case of mixed labyrinth fractals \cite[Theorem 1]{mixlaby}.

\begin{theorem}\label{theo:dendrite}
Every (self-similar or mixed) labyrinth fractal is a dendrite.
\end{theorem}

 We recall that a dendrite is a connected compact Hausdorff space, that is locally connected and contains no simple closed curve.
By Theorem \ref{theo:dendrite} and a result in Kuratowski's book \cite[Corollary~2, p. 301]{Kuratowski} it follows that between any pair of points $x\neq y$ in the dendrite $L_{\infty}$ there is a unique arc in the fractal that connects them.

It was shown \cite{laby_4x4,laby_mxm,mixlaby,arcs_mixlaby} that certain patterns, called blocked labyrinth patterns, generate fractals with interesting properties concerning the lengths of arcs in the fractal. More precisely, a labyrinth pattern is called \emph{horizintally blocked}, if in the row that contains the horizontal exit pair of the pattern there exists at least one black square. Analogously, a labyrinth pattern is \emph{vertically blocked}, if there exists at least one black square in the column that contains the vertical exit pair of the pattern. Figures \ref{fig:A1A2}, \ref{fig:W2} and \ref{fig:ex:special_cross} show examples of horizontally and vertically blocked labyrinth patterns.

In the self-similar case the following facts  were proven \cite[Theorem 3.18]{laby_mxm}:
\begin{theorem} \label{theo:main_self-similar}Let $L_{\infty}$ be the (self-similar) labyrinth fractal generated by a horizontally and 
vertically blocked $m\times m$-labyrinth pattern. 
Between any two points in $L_{\infty}$ there is a unique arc ${a}$.
The length of ${a}$ is infinite.
The set of all points, at which no tangent to ${a}$ exists, 
is dense in ${a}$.
\end{theorem}
Moreover, it was proven, that in this case the box counting dimension of all arcs in the fractal is strictly greater than $1$. Moreover, these arcs all have box counting and Hausdorff dimension strictly less than $2$, which easily follows, e.g., from the fact that the dimension (both box counting and Hausdorff) of the whole fractal is strictly less than $2$, see \cite{laby_4x4, laby_mxm}. 

In the mixed case, it was shown \cite{arcs_mixlaby} that the fact that the patterns are blocked is not sufficient in order to obtain all arcs with infinite length. 

Very recent research \cite[Theorem 5]{supermix} reveals the following  sufficient but not necessary condition in order to obtain a mixed labyrinth fractal where all (non-trivial) arcs in the fractal have infinite length.

\begin{theorem}\label{theo:sufficient_cond_infinite_arcs}
 If the width-sequence $\{m_k\}_{k\ge 1}$ of the sequence $\{\Aa_k\}_{k\ge 1}$ of both horizontally and vertically blocked labyrinth patterns satisfies the condition
\begin{equation}\label{eq:sufficient_cond_infinite_arcs}
\sum_{k\ge 1}^{\infty}\frac{1}{m_k}=\infty,
\end{equation}
then the mixed labyrinth fractal $\linf$ has the property that for any two distinct points $x,y\in \linf$ the length of the arc in $\linf$ that connects $x$ and $y$ is infinite.
\end{theorem}

Below we recall a result \cite[Lemma 4]{mixlaby} that enables us to use paths in the graphs of labyrinth patterns in order to construct arcs in a labyrinth fractal $\linf$, both in the self-similar, and in the mixed case.
\begin{lemma}\label{lemma:arc_construction}  (Arc Construction) Let $a,b\in L_{\infty}$, where $a\neq b$. For all $n \ge 1$, there are $W_n(a),W_n(b)\in V(\mathcal{G}({\cal W}_{n}))$ such that 
\begin{itemize}
\item[(a)]$W_1(a)\supseteq W_2(a)\supseteq\ldots$,
\item[(b)]$W_1(b)\supseteq W_2(b)\supseteq\ldots$,
\item[(c)]$\{a\}=\bigcap_{n=1}^{\infty}W_n(a)$,
\item[(d)]$\{b\}=\bigcap_{n=1}^{\infty}W_n(b)$.
\item[(e)]The set $\bigcap_{n=1}^{\infty}\left(\bigcup_{W\in p_n(W_n(a),W_n(b))} W\right)$ is an arc between $a$ and $b$. 
\end{itemize}
\end{lemma}

Until now, only one result with respect to fractal dimension of mixed labyrinth fractals was proven \cite[Proposition 6, p.121]{mixlaby}:
\begin{proposition}  
If $a$ is an arc between the top and the bottom exit in $L_{\infty}$ then 
\[
\liminf_{n\rightarrow \infty} \frac{\log(\A(n))}{\sum_{k=1}^{n}\log (m_k)}= \underline{\dim}_B(a)\le\overline{\dim}_B(a)= \limsup_{n\rightarrow \infty} \frac{\log(\A(n))}{\sum_{k=1}^{n}\log (m_k)},
\]
where $\A(n)$ denotes the number od squares in the path in $\G(\W_n)$ that connects the top and the bottom exit of $\W_n$.
For the other pairs of exits, the analogue statement holds. 
\end{proposition}

In the considerations to follow we aim to construct mixed labyrinth fractals whose box counting dimension can be established. 
First, we introduce a special sequence of labyrinth patterns, called snake cross patterns, which enables us to construct mixed labyrinth fractals with the property that the box counting dimension of any non-trivial arc (i.e., arc that connects two distinct points of the fractal) equals $2$.
Subsequently, we give a method for constructing mixed labyrinth fractals of any given (rational) box counting dimension between $1$ and $2$.
For more details on the box counting dimension and its properties we refer, e.g., to Falconer's book \cite{falconer_book}.

\section{Snake cross patterns and mixed labyrinth sets, and the fractals generated by them}\label{sec:snake_cross_patterns}
In recent research \cite{arcs_mixlaby}, a special family of labyrinth patterns, called ``special cross patterns'' were used in order to construct mixed labyrinth fractals with certain properties of their arcs. These patterns were chosen due to their symmetry properties, which make them convenient tools when studying lengths of arcs between exits of the fractal. Here we introduce a family of patterns called \emph{snake cross patterns}, whose name is due to the fact that their shape and some of their properties remind us of the special cross patterns, but, since here the length of the four ``arms'' of the cross is meant to be  as long as possible, they have in general more and much bigger windings than the special cross patterns. Figure \ref{fig:ex:special_cross} shows a speciall cross pattern, and in Figures \ref{fig:A1A2A3},\ref{fig:A4A5} and \ref{fig:colorA2} we see snake cross patterns.

\begin{figure}[h!]
\begin{center}
\begin{tikzpicture}[scale=.15]
\draw[line width=1.5pt] (0,0) rectangle (22,22);
\filldraw[fill=gray!30, draw=black] (10,10) rectangle (12, 12);
%---- jetzt das dicke raster--------
\draw[line width=0.5pt] (2, 0) -- (2,22);
\draw[line width=0.5pt] (4, 0) -- (4,22);
\draw[line width=0.5pt] (6, 0) -- (6,22);
\draw[line width=0.5pt] (8, 0) -- (8,22);
\draw[line width=0.5pt] (10, 0) -- (10,22);
\draw[line width=0.5pt] (12, 0) -- (12,22);
\draw[line width=0.5pt] (14, 0) -- (14,22);
\draw[line width=0.5pt] (16, 0) -- (16,22);
\draw[line width=0.5pt] (18, 0) -- (18,22);
\draw[line width=0.5pt] (20, 0) -- (20,22);
\draw[line width=0.5pt] (0, 2) -- (22,2);
\draw[line width=0.5pt] (0, 4) -- (22,4);
\draw[line width=0.5pt] (0, 6) -- (22,6);
\draw[line width=0.5pt] (0, 8) -- (22,8);
\draw[line width=0.5pt] (0, 10) -- (22,10);
\draw[line width=0.5pt] (0, 12) -- (22,12);
\draw[line width=0.5pt] (0, 14) -- (22,14);
\draw[line width=0.5pt] (0, 16) -- (22,16);
\draw[line width=0.5pt] (0, 18) -- (22,18);
\draw[line width=0.5pt] (0, 20) -- (22,20);
%--- jetzt die schwarzen quadrate-------------
\filldraw[fill=black, draw=black] (0,0) rectangle (8, 10);
\filldraw[fill=black, draw=black] (8,8) rectangle (10, 10);
\filldraw[fill=black, draw=black] (8,0) rectangle (10, 2);
\filldraw[fill=black, draw=black] (4,10) rectangle (6, 12);
\filldraw[fill=black, draw=black] (0,14) rectangle (10, 22);
\filldraw[fill=black, draw=black] (0,12) rectangle (2, 17);
\filldraw[fill=black, draw=black] (8,12) rectangle (10, 14);
\filldraw[fill=black, draw=black] (10,16) rectangle (12, 18);
\filldraw[fill=black, draw=black] (12,20) rectangle (14, 22);
\filldraw[fill=black, draw=black] (12,12) rectangle (14, 14);
\filldraw[fill=black, draw=black] (14,12) rectangle (22, 22);
\filldraw[fill=black, draw=black] (12,8) rectangle (14, 10 );

\filldraw[fill=black, draw=black] (16,10) rectangle (18, 12);
\filldraw[fill=black, draw=black] (12,0) rectangle (22, 8);
\filldraw[fill=black, draw=black] (20,8) rectangle (22, 10);
\filldraw[fill=black, draw=black] (10,4) rectangle (12, 6);

\end{tikzpicture}
\caption{An example: a special cross pattern with width $11$}
\label{fig:ex:special_cross}
\end{center}
\end{figure}
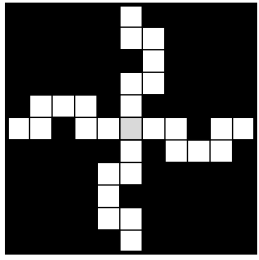

\begin{figure}[hhhh]
\begin{center}
\includegraphics[scale=0.7]{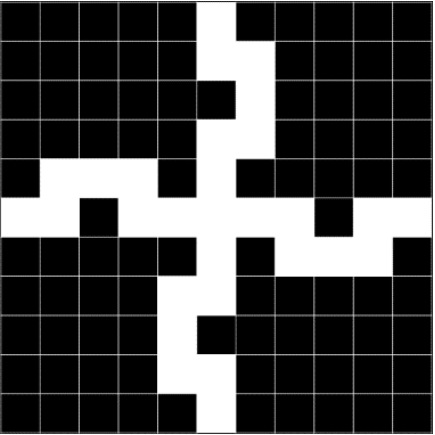}~~
\includegraphics[scale=0.7]{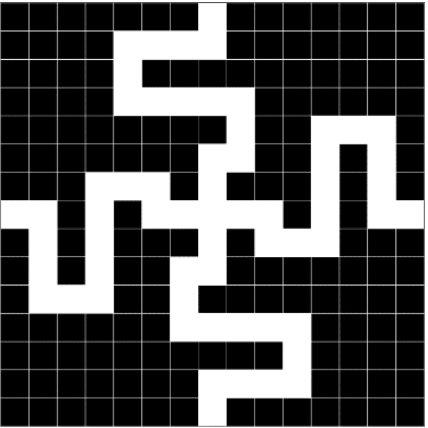}~~
\includegraphics[scale=0.7]{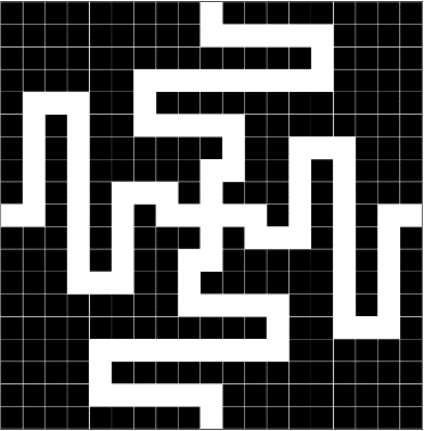}
\caption{The snake cross patterns $\Aa_1, \Aa_2, \Aa_3$ of width $11$, $15$, and $19$, respectively.} 
\label{fig:A1A2A3}
\end{center}
\end{figure}

\begin{figure}[!]
\begin{center}
\includegraphics[scale=0.7]{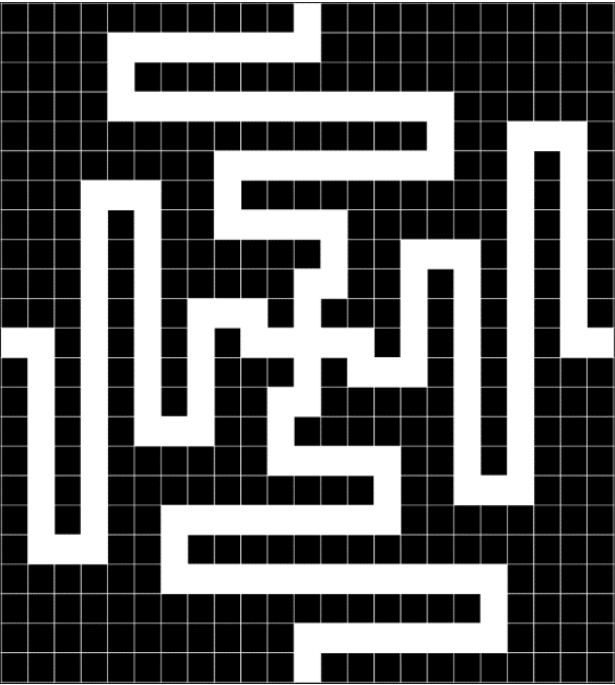}~~~~
\includegraphics[scale=0.7]{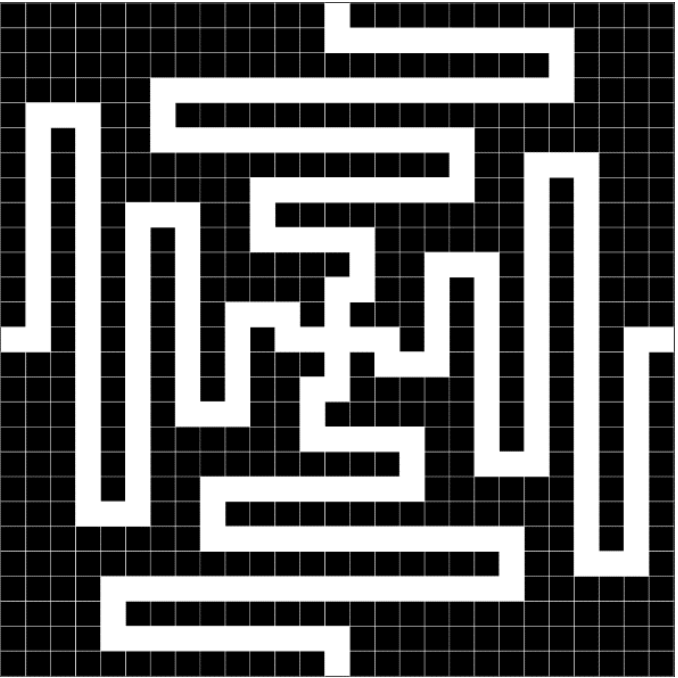}
\caption{The snake cross patterns $\Aa_4, \Aa_5$ of width $23$, and $27$, respectively. 
} 
\label{fig:A4A5}
\end{center}
\end{figure}

In order to define the sequence $\{ \Aa_k\}_{k\ge 1}$ of snake cross patterns, with corresponding width sequence $\{ m_k\}_{k\ge 1}$, let us first  set $m_k:=4k+7$, for all $k\ge 1$. Each pattern has a central square $C_k$, namely the white square that lies in row $2k+4$ and column $2k+4$, and which is coloured in white. Moreover, the top and bottom exit of any snake cross pattern lie in column $2k+4$, and the left and right exit lie in row $2k+4$.
In Figure \ref{fig:ex:special_cross} the central square of the pattern is coloured in light grey, in Figure \ref{fig:colorA2} the central square of the pattern is white (for better visibleness on the coloured figure).
 Four ``arms'' emerge from the central square. In Figure \ref{fig:colorA2} these are marked with different colours.
 The arms' shape, the windings on each of the four arms of the cross, resemble snakes in movement. The four arms are identical (up to rotation), and each of them can be obtained by the same procedure:
\begin{enumerate}
\item Start from the central square $C_k$.
\item The exit square on one of the sides of the unit square is selected as a distinguished vertex of $\G(\Aa_k)$, namely, the midpoint of one of its sides coincide with the midpoint of the mentioned side of the unit square.
\item  ``Move two steps'' (two neighbouring squares) towards the selected exit.
\item Turn to the right and move through pairwise neigbouring squares until either a diagonal (see Figure \ref{fig:colorA2}) is reached or a neighbour of the selected exit is reached. In the latter case stop the construction of the arm.
\item Turn to the left and ``move two steps''.
\item Turn to the left and move until either a diagonal is reached or a neighbour of the selected exit is reached. In the latter case stop the construction of the arm.
\item Turn to the right and ``move two steps''.
\item Repeat steps $4$ to $7$.
\end{enumerate}

\begin{figure}[h!]
\begin{center}
\includegraphics[scale=0.3]{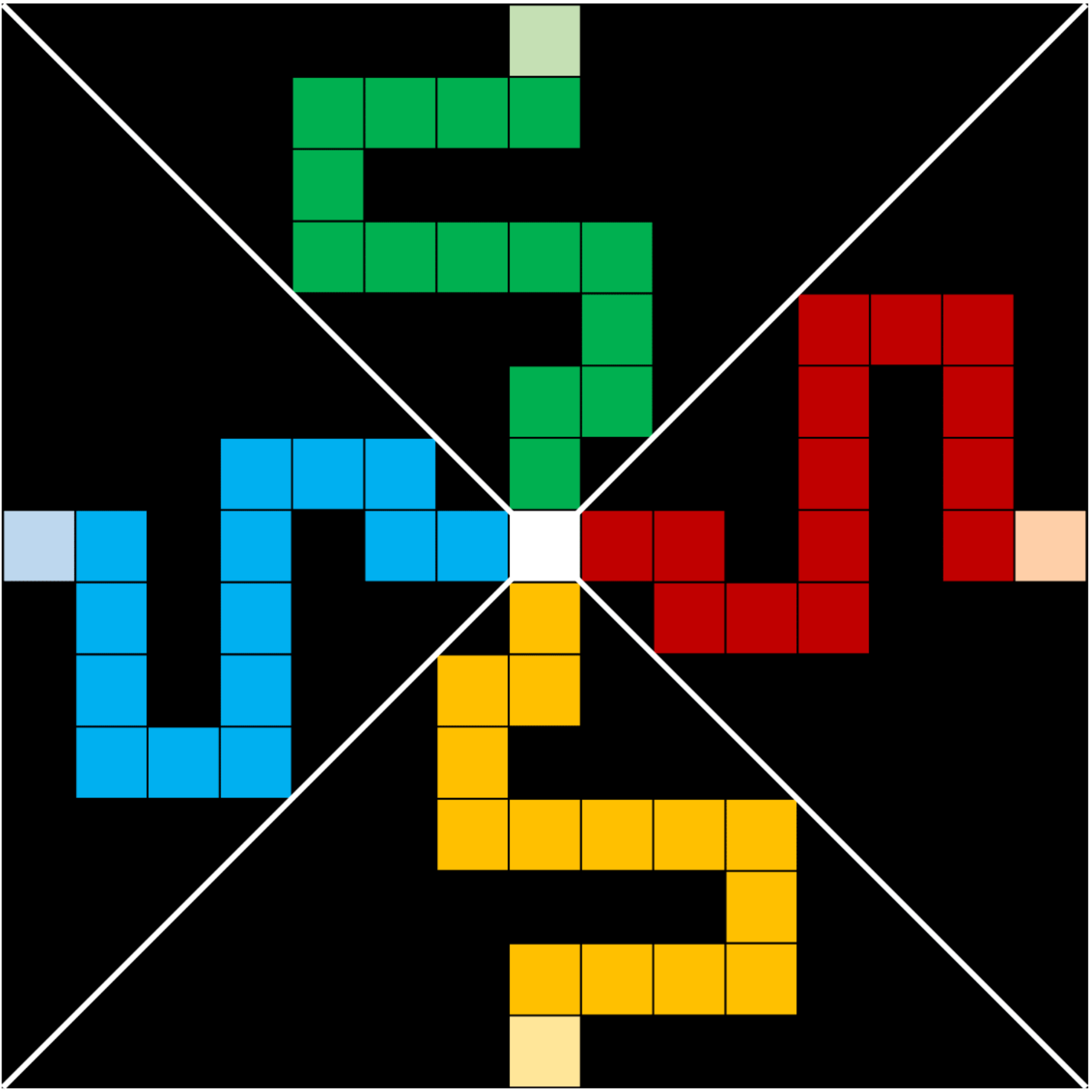}
\caption{The four arms, the exits and the central square in the snake cross pattern $\Aa_2$} 
\label{fig:colorA2}
\end{center}
\end{figure}

\begin{remark}
The $k$-th snake cross pattern has $k$ windings (``detours'') on each of its four arms and $2k$ black squares in the row  that contains its left and right exit, as well as in the column that contains its top and botom exit. We note, according to Figures \ref{fig:ex:special_cross} and \ref{fig:A1A2A3}, that the first snake cross pattern $\Aa_1$ is in particular a special cross pattern.
\end{remark}

It is easy to see that by replacing ``left'' with ``right'' and  viceversa in the above construction algorithm of a snake cross pattern we obtain an other labyrinth pattern. Moreover, due to its shape, we can see it just as the ``reflected'' initial pattern. We call the initial one (defined above) a \emph{right snake cross pattern} because the first turn was to the right, and its corresponding ``reflected'' pattern a \emph{left snake cross pattern}. Throughout this article we assume, without loss of generality, that all snake cross patterns mentioned below are right snake cross patterns.

 One can immediately see that the snake cross patterns are both horizontally and vertically blocked labyrinth patterns. One can use these patterns in order to define a mixed labyrinth fractal $\linf$: let ${\Aa_k}_{k\ge 1}$ be a sequence of snake cross patterns (as defined above) and $\linf$ the corresponding labyrinth fractal. We call $\linf$ a \emph{mixed snake labyrinth fractal}.

Based on Theorem \ref{theo:dendrite} and Theorem \ref{theo:sufficient_cond_infinite_arcs} one immediately obtains the proof of the following result.

\begin{proposition}\label{prop:properties_snakelabyfractal}
Let $\linf$ be a mixed snake labyrinth fractal. Then:
\begin{enumerate}
\item $\linf$ is a dendrite.
\item For any distinct points $x,y \in \linf$, the arc in $\linf$ that connects $x$ and $y$ has infinite length.
\end{enumerate}
\end{proposition}

\section{Mixed labyrinth fractals with arcs of maximal box counting dimension}\label{sec:mixed_arcs_max_dim}

Now, our aim is to find the box counting dimension of arcs in a mixed snake labyrinth fractal. As a first step, we are interested in the arcs that connect exits of $\linf$. Therefore, let us consider the lengths of the paths between exits in the labyrinth sets $\W_n$, for $n\ge 1$. For $n\ge 1$, let $\A(n)$ denote the length of the path in $\G(\W_n)$ that connects the top and bottom exit of the labyrinth set $\W_n$, and, analogously use the notations $\B(n), \C(n), \D(n), \E(n), \F(n)$ for the lengths of the paths connecting the other six pairs of exits, respectively, as shown by the symbols. Moreover, for $k\ge 1$, let $\A_k, \B_k, \C_k, \D_k, \E_k, \F_k$ denote the lengths of the corresponding paths in $\G(\Aa_k)$. By the lengths of paths in the graphs we mean the number of graph vertices (i.e., the number of white squares) that lie on the path.

\begin{lemma}\label{lemma:arms_of_same_length}
With the above notations,
\[
\A(n)=\B(n)=\C(n)=\D(n)=\E(n)=\F(n)=:D(L_n), \text { for all } n\ge 1.
\]
\end{lemma}
\begin{proof}(\emph{Sketch})
One can immediately see that $\A_k= \B_k= \C_k= \D_k= \E_k= \F_k$, for all $k\ge 1$, and thus in particular $\A(1)=\B(1)=\C(1)=\D(1)=\E(1)=\F(1)$. By induction, one can then immediately obtain the above result.
\end{proof}

We remark that one could also prove the above equalities by using the results on path matrices of mixed labyrinth fractals \cite[Proposition 1]{mixlaby}, but here in our special case it is not necessary to use the path matrices. \\

\begin{proposition}\label{prop:dimB_arcs_exits=2}
Let $\linf$ be the mixed labyrinth fractal generated by the sequence of snake cross patterns $\{\Aa_k \}_{k\ge 1}$ defined above. If $a$ is an arc that connects two distinct exits of $\linf$, then its box counting dimension is $\dim_B(a)=2$. 
\end{proposition}

\begin{proof}
In order to compute the box counting dimension of the arc between two distinct exits of $\linf$, we use an alternative definition of the box counting dimenstion \cite[Definition 1.3]{falconer_book}, and we apply the property that one can use, instead of $\delta \to 0$, apropriate sequences $(\delta_k)_{k\ge 0}$ (see the cited reference). In our case, for $\delta_k=\frac{1}{m(k)}$, we have $\delta_k < \frac{1}{5}\delta_{k-1}$, and thus the box counting dimenstion of any  arc between two exits in the fractal is given by 
\begin{equation*}
\dim_B(\linf)=\lim_{n\to \infty} \frac{\log D(L_n)}{\log m(n)}.
\end{equation*}
Therefore, let us now estimate $D(L_n)$. In order to do this, we count the number of squares that are along one of the four arms of the pattern $\Aa_k$ (we do not include the central square of the pattern), taking into account the following facts, according to the procedure described in Section \ref{sec:snake_cross_patterns}: 
\begin{itemize}
\item[($i$)] in the first step of the construction of the arm one square is added;
\item[($ii$)] in step $3$ two squares are added;
\item[($iii$)] in the non-final steps $4$ and $6$ the number of added squares is $$1+4+8+\dots+4(k-1)=1+4\frac{(k-1)k}{2}=2k^2-2k+1;  $$
\item[($iv$)] in the steps $3$, $5$ and $7$ there are $2k$ added squares;
\item[($v$)] in the final step $4$ or $6$ there are $2k-1$ added squares.
\end{itemize}

The above counting yields a total number of $$1+2+(2k^2-2k+1)+2k+(2k-1)=2k^2+2k+3$$ squares. Hence, the length of the path in $\G(\Aa_k)$ between two exits of the snake cross pattern $\Aa_k$ is $2(2k^2+2k+3)+1=4k^2+4k+7$. Thus we obtain
$$ D(L_n)=\prod_{k=1}^n\left( 4k^2+4k+7 \right)$$
for the number of squares in the path  in $\G(\W_n)$ that connects two distinct exits of the labyrinth set $\W_n$.
Due to the arc construction given in Lemma \ref{lemma:arc_construction} and the fact that the arc between two exits can be obtained as the limit of the nested sequence of sets defined, for each $n\ge1$, as the union of the squares that lie in the path in $\G(\W_n)$ between the corresponding exits of the labyrinth set $\W_n$, and therefore is covered by the squares in this path, we obtain
\begin{equation}\label{eq:dimB_arcs_exits}\displaystyle
\dim_B(a)=\lim_{n\to \infty} \frac{\log\left(\prod_{k=1}^n (4k^2+4k+7) \right)}{\log \left(\prod_{k=1}^{n} (4k+1)\right)}=\lim_{n\to \infty}\frac{\sum_{i=1}^n\log(4k^2+4k+7)}{\sum_{k=1}^n \log(4k+1)}
\end{equation} 
Note that for every $\varepsilon>0$ and for sufficiently large $k$ the inequality 
$4k^2+4k+7>(4k+1)^{2-\varepsilon}$ holds, hence $\dim_B(a)\ge 2.$ On the other hand, $4k^2+4k+7<(4k+1)^2$, which then implies $\dim_B(a)=2.$
\end{proof}

Proposition \ref{prop:dimB_arcs_exits=2}, the structure of  mixed labyrinth fractals defined by snake cross patterns, and the properties of the box counting dimension immediately yield the following corollary.

\begin{corollary}\label{cor:dimB(fractal)=2}
Let $\linf$ be a mixed labyrinth fractal generated by a sequence of snake cross patterns as above. Then $\dim_B(\linf)=2$.
\end{corollary}

{\bf Notation.}
Let $x \in \linf$ and $n\ge 1$. Along this paper the notation $W_n(x)$ has the following meaning (analogoulsy as introduced  in the proof of \cite[Lemma 6]{laby_4x4} for the case of self similar labyrinth fractals ): Let $\W(x)$ be the set of all white squares in $\bigcup_{n\ge 1}\W_n$ that contains the point $x$. Let $W_1(x)$ be a white square in $\W_1$ that contains infinitely many (white) squares of $\W(x)$ as a subset. For $n\ge 2$ we define $W_n(x)$ as a white square in $\W_n$, such that $W_n(x)\subseteq W_{n-1}(x)$, and $W_n(x)$ contains infinitely many white squares of $W(x)$ as a subset.

\begin{theorem}\label{theo:all_arcs_dimB=2}
Let $\linf$ be a mixed labyrinth fractal generated by a sequence of snake cross patterns as above. Then, for any  distinct points $x,y\in \linf$ the arc $a(x,y)$ that connects them in $\linf$ has box counting dimension  $\dim_B(a(x,y))=2$.
\end{theorem}

\begin{proof}
Let $x\ne y$ be two distinct points in the fractal $\linf$.
From $a(x,y)\subset \linf$ it imediately follows that $\dim_B(a(x,y))\le \dim_B \linf=2.$ 
In order to prove the converse inequality, we proceed as follows.

For any $n\ge 1$ let $W_n(x)$ and $W_n(y)$ be white squares of level $n$, corresponding to the notation introduced above.

According to the construction of labyrinth sets and fractals it immediately follows that there exists a value $n_0\ge 1$ such that the path in the tree $\G(\W_{n_0})$ that connects $W_{n_0}(x)$ and $W_{n_0}(y)$ consists of at least three white squares (vertices of $\G(\W_{n_0})$), i.e., there exists a square $W\in \W_{n_0}$ in this path, such that $x,y$ do not lie in the interior of $W$. 

Now, let us note that, due to the construction and structure of mixed labyrinth fractals, the intersection of the arc $a(x,y)$ that connects $x$ and $y$ in $\linf$  with the swuare $W$ is an arc $a'=a(x,y)\cap W$ that connects  two exits of the fractal $\linf'=\linf \cap W$, which is just a scaled labyrinth fractal, obtained from the sequence of snake cross patterns $\{\Aa'_k\}_{k\ge 1}$, with $\Aa'_k=\Aa_{k+n_0}$, where $n$ was chosen and fixed above. Thus, the arc $a'$ is just the image of an arc between two exits of $\linf$ through a similarity mapping of factor $1/m_1 m_2 \dots m_{n_0}$.
Now we sum up as in formula \eqref{eq:dimB_arcs_exits}, but for the integer $k$ ranging from $n_0+1$ (instead of $1$) to $n$. Since the limit of a sequence does not change by removing the first $n_0$ elements for some fixed $n_0$ we obtain,  as in Proposition \ref{prop:dimB_arcs_exits=2}, that 
any arc that connects distinct exits of $\linf'$ has box counting dimension two, and thus, in particular,  $\dim_B(a')= 2$, and since $\dim_B(a(x,y))\ge \dim_B(a')=2$, we get $\dim_B(a(x,y))= 2$. 
\end{proof}

\begin{remark}
The above results show that the snake cross patterns enable us to construct 
(mixed) labyrinth fractals based on horizontally and vertically blocked labyrinth patterns where the box counting dimension of every arc in the fractal as well as that of the fractal itself equals two. In the case of self-similar labyrinth fractals \cite{laby_4x4, laby_mxm} generated by both horizontally and vertically blocked patterns, all arcs in the fractal have the same box counting dimension (strictly less than $2$) and this is in general different from de dimension of the fractal. 
\end{remark}

\section{``Omnidimensionality'' of mixed labyrinth fractals}\label{sec:omnidimensional}
In this section we show that one can use the snake cross patterns in order to construct mixed labyrinth fractals of any dimension.

\begin{theorem}\label{theo:omnidimensional}
Let $\delta \in [1,2]$. Then there exists a (mixed) labyrinth fractal $\linf$ with the property that every two exits in the fractal are connected in $\linf$ by an arc whose box counting dimension equals $\delta$.
\end{theorem}

Before giving a proof of the theorem, we prove the following useful result.
\begin{lemma}\label{lemma:exist_sequences}
Let $a,b,c,d$ be positive real numbers and let $\frac{a}{b}< \alpha< \frac{c}{d}$. Then there are sequences $(p_k)_{k\ge 1}$ and $(q_k)_{k\ge 1}$ of integers such that the sequence $(r_k)_{k\ge 1}$, with $r_k=\frac{p_k\cdot a+q_k\cdot c}{p_k \cdot b+q_k \cdot d}$, is increasing and $\lim_{k\to \infty} r_k=\alpha $.
\end{lemma}
\begin{proof}
Note that from the lemma's assumptions we immediately have $\frac{c-\alpha d}{\alpha b -a}>0$. Hence, there is an increasing sequence $\left (\frac{p_k}{q_k}\right)_{k\ge 1}$,  with $p_k, q_k$ positive integers, such that $\lim_{k\to \infty}\frac{p_k}{q_k} =\frac{c-\alpha d}{\alpha b -a}$.  We denote $\varepsilon_k:=\frac{p_k}{q_k}-\frac{c-\alpha d}{\alpha b -a}$, for $k\ge 1$. Remark, that $(\varepsilon_k)_{k\ge 1}$ is a decreasing sequence that converges to zero. Now, let us proceed to the following computations:
\begin{align*}
r_k&=\frac{p_ka+q_kc}{p_k b +q_kd}=\frac{\frac{p_k}{q_k}a+c}{\frac{p_k}{q_k}b+d}= 
\frac{\left( \frac{c-\alpha d}{\alpha b-a} +\varepsilon_k\right)a+c}{\left( \frac{c-\alpha d}{\alpha b-a} +\varepsilon_k \right)b+d}=\frac{(c-\alpha d)a+(\alpha b -a)a\varepsilon_k+c(\alpha b-a)}{(c-\alpha d)b+(\alpha b -a)b\varepsilon_k+d(\alpha b-a)}\\
&=\frac{(cb-da)\alpha+(\alpha b-a)a\varepsilon_k}{(cb-da)+(\alpha b-a)b\varepsilon_k}=
\frac{(cb-da)\alpha+(\alpha b-a)b\alpha \varepsilon_k+[(\alpha b-a)a-(\alpha b-a)b \alpha]\varepsilon_k}{(cb-da)+(\alpha b-a)b \varepsilon_k} \\
&=\alpha - \frac{(\alpha b-a)^2}{(cb-da)+(\alpha b-a)b\varepsilon_k}\varepsilon_k.
\end{align*} 
Thus, $\lim_{k\to \infty}r_k= \alpha$. In order to prove that the sequence $(r_k)_{k\ge 1}$ is increasing, let us consider the function 
\[f(x):=\alpha-\frac{(\alpha b-a)^2}{(cb-da) + (\alpha b -a)bx}x.\]
Its first derivative is 
\[
f'(x)=-\frac{(\alpha b -a)^2(cb-da)}{\big((cb-da) + (\alpha b -a)bx\big)^2}<0.
\] Thus, the fact that the sequence $(\varepsilon_k)_{k\ge 1}$ is decreasing  yields  that the sequence $(r_k)_{k\ge 1}$ is increasing, which completes the proof. 
\end{proof}
\begin{proof}[Proof of Theorem \ref{theo:omnidimensional}]
The case $\delta=2$ has been solved in Section \ref{sec:mixed_arcs_max_dim}.

For the case $\delta=1$ it is enough to consider, e.g., the mixed labyrinth fractal generated by a sequence of labyrinth patterns $\{\Aa_k  \}_{k\ge 1}$ defined as follows: Figure \ref{fig:ex:cross} shows the pattern $\Aa_1$. In general, for $k\ge 1$, $\Aa_k$ has width $2k+1$ and has the shape of a cross, consisting only of the $(k+1)$-th row and $(k+1)$-th column.  Since the arcs connecting the exits of the resulting mixed labyrinth fractal are either straight line segments of length $1$ or the union of two straigh line segements of length $1/2$ (each), we are done. In this case it easily follows $\dim_B(\linf)=\dim_H (\linf)=1$ and 
box counting dimension of any arc in the dendrite is $1$.

\begin{figure}[h!]
\begin{center}
\begin{tikzpicture}[scale=.35]
\draw[line width=1.2pt] (0,0) rectangle (9,9);
%\filldraw[fill=gray!30, draw=black] (3,3) rectangle (6, 6);
%---- jetzt das dicke raster--------
\draw[line width=1pt] (3, 0) -- (3,9);
\draw[line width=1pt] (6, 0) -- (6,9);

\draw[line width=1pt] (0, 3) -- (9,3);
\draw[line width=1pt] (0, 6) -- (9,6);

%--- jetzt die schwarzen quadrate-------------
\filldraw[fill=black, draw=black] (0,0) rectangle (3, 3);
\filldraw[fill=black, draw=black] (0,6) rectangle (3, 9);
\filldraw[fill=black, draw=black] (6,6) rectangle (9, 9);
\filldraw[fill=black, draw=black] (6,0) rectangle (9, 3);

\end{tikzpicture}
\caption{The first pattern in the sequence $\{ \Aa_k \}_{k\ge 1}$ that generates a mixed labyrinth fractal having all arcs of box counting dimension one.}
\label{fig:ex:cross}
\end{center}
\end{figure}
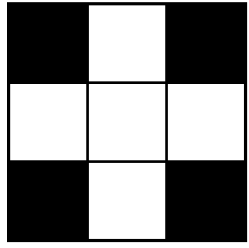

Now, we assume $\delta \in (1,2)$. Let $\Aa_0$ denote the (unblocked) labyrinth pattern in Figure \ref{fig:ex:cross}. For $k\ge 1$ let $\Aa_k$ be the $k$-th snake cross pattern, as defined in Section \ref{sec:snake_cross_patterns}. For every $k\ge 1$ we introduce the \emph{$k$-th  dimension quotient} $d_k$ as follows. By definition,

 \[d_0:=1 \text{ and } d_k:= \frac{\log(4k^2+4k+7)}{\log(4k+1)}, \text{ for all } k\ge 1.\]

Note that the quotients $d_k$, for $k\ge 1$ occur in the computations in formula \eqref{eq:dimB_arcs_exits}. 
Since $\lim_{k\to \infty}d_k=2$, it follows that there is an integer $k \ge 1$ such that 
	\[ d_k\le \delta \le d_{k+1} .\]
If $\delta=d_k$, for some $k\ge 1$ then the (self similar) labyrinth fractal constructed based on the snake cross pattern $\Aa_k$ is a solution, and if $\delta =d_{k+1}$ then the (self similar) labyrinth fractal constructed based on the snake cross pattern $\Aa_{k+1}$ is a solution. Let us now assume that $d_k<\delta <d_{k+1}.$ More precisely, we then have for this value of $k$,
\begin{equation}
\label{eq:ineq_dim_quotients}
\frac{\log(4k^2+4k+7)}{\log(4k+1)}<\delta< \frac{\log \left( (4(k+1)^2+4(k+1)+7) \right)}{\log \left(4(k+1)+1 \right)}.
\end{equation}
We use the following notation: 
\[
a=\log(4k^2+4k+7), \;b=\log(4k+1), \; c=\log\big(4(k+1)^2+4(k+1)+7 \big), 
\; d= \log\big(4(k+1)+1  \big).
\]
Moreover,  let $\alpha=\delta$ and let $(p_k)_{k\ge 1}$ and $(q_k)_{k\ge 1}$ be the sequences occurring in Lemma \ref{lemma:exist_sequences}. 

We introduce a new sequence of labyrinth patterns $(\Aa_k^*)_{k\ge 1}$ in the following way. For $k\ge 1$, let $\Aa_k^*(\delta)$, in short, when $\delta$ is known, $\Aa_k^*$, be the labyrinth pattern of width $m^*_k=m_k^{p_k}m_{k+1}^{q_k}$ that is identical to the mixed labyrinth set (of level $p_k+q_k$) constructed by first applying $p_k$ times the pattern $\Aa_k$ and then applying $q_k$ the pattern $\Aa_{k+1}$. Let $\linf(\Aa^*(\delta))$ be the mixed labyrinth fractal defined by the sequence of patterns $(\Aa_k^*)_{k\ge 1}$. Let $a(e_1,e_2)$ be an arc in $\linf(\Aa^*(\delta))$ that connects two exits $e_1\ne e_2$ of this mixed labyrinth fractal. Then its box counting dimension is
\begin{align*} 
\dim_B(a(e_1,e_2))&= \lim_{n\to \infty}
\frac{\log \left( \prod_{k=1}^n\big( (4k^2+4k+7)^{p_k}\left( 4(k+1)^2+4(k+1)+7 \right)^{q_k}   \big) \right) }{\log  \left( \prod_{k=1}^n\big( (4k+1)^{p_k}\left( 4(k+1)+1 \right)^{q_k}   \big) \right)   }\\
&= \lim_{n\to \infty}\frac{\sum_{k=1}^n p_ka+q_kc}{\sum_{k=1}^n p_kb+q_kd}.
\end{align*}
Thus, with the notation from Lemma \ref{lemma:exist_sequences} we obtain
\begin{align*} 
\dim_B(a(e_1,e_2))&= \lim_{n\to \infty} \frac{\sum_{k=1}^n (p_kb+q_kd)
\left( \alpha - \frac{(\alpha b -a)^2}{(cb-da)+(\alpha b-a)b \varepsilon_k}\varepsilon_k  \right)}{\sum_{k=1}^n(p_kb+q_kd)}\\
&=\lim_{n\to \infty} 
\left( 
\alpha- 
\frac{\sum_{k=1}^n (p_kb+q_kd)\cdot \frac{(\alpha b -a)^2}{(cb-da)+(\alpha b-a)b \varepsilon_k }\varepsilon_k}
{\sum_{k=1}^n(p_kb+q_kd)} 
\right ).
\end{align*} 
Hence, $\dim_B(a(e_1,e_2))\le \alpha.$ 
It remains to prove that for every $\varepsilon >0$, we have $\dim_B(a(e_1,e_2))> \alpha-\varepsilon,$ i.e., that
\begin{equation}\label{eq:inequality}
\lim_{n\to \infty}\frac{\sum_{k=1}^n (p_kb+q_kd)\cdot \frac{(\alpha b -a)^2}{(cb-da)+(\alpha b-a)b \varepsilon_k }\varepsilon_k}
{\sum_{k=1}^n(p_kb+q_kd)} <\varepsilon.
\end{equation}
By taking into consideration that 
$\displaystyle
\lim_{k\to \infty}\frac{(\alpha b-a)^2}{(cb-da)+(\alpha b -a)b\varepsilon_k}\varepsilon_k =0, 
$
and that the sequence $\left( \frac{(\alpha b-a)^2}{(cb-da)+(\alpha b -a)b\varepsilon_k}\varepsilon_k \right)_{k\ge 1}$ is monotonically decreasing, it follows that there is a $k_0$ such that for all $k\ge n_0$ the inequality $ \displaystyle \frac{(\alpha b-a)^2}{(cb-da)+(\alpha b -a)b\varepsilon_k}\varepsilon_k <\frac{\varepsilon}{2}$ holds, and therefore we obtain that the left hand term in \eqref{eq:inequality} is less than
\[
\lim_{n\to \infty} \frac{\sum_{k=1}^{n_0} (p_kb+q_kd)\cdot \frac{(\alpha b -a)^2}{(cb-da)+(\alpha b-a)b \varepsilon_k }\varepsilon_k+\frac{\varepsilon}{2} \cdot \sum_{k=n_0+1}^n(p_k b+q_k d)}{\sum_{k=1}^{n_0}(p_k b+q_k d)+ \sum_{k=n_0+1}^{n}(p_k b+q_k d)} <\frac{\varepsilon}{2},
\]
which yields $\dim_B(a(e_1,e_2))=\alpha =\delta.$

\end{proof}

\begin{corollary}\label{cor:dimB(fractal=delta)}
Under the assumptions of Proposition \ref{prop:dimB_arcs_exits=2} we have 
\[
\dim_B\big(\linf(\Aa^*(\delta))\big)=\delta.
\]
\end{corollary}

\begin{theorem}\label{theo:any_arc_dimB=delta}
Let $\delta \in [1,2)$, and $\linf(\Aa^*(\delta))$ be be the corrensponding labyrinth fractal, constructed as in the proof of Theorem \ref{theo:omnidimensional}. Then,  
for any two distinct points $x,y \in \linf$ the arc $a(x,y)$ that connects the two points  in the fractal has $\dim_B(a(x,y))=\delta.$
\end{theorem}

\begin{proof} (Sketch.) Let $\delta$ and $\linf(\Aa^*(\delta))$ be as in the above assumptions. 
For any distinct points $x,y \in \linf(\Aa^*(\delta))$, we have for the corresponding arc $a(x,y)\subseteq \linf(\Aa^*(\delta))$ that $\dim_B(a(x,y)) \le \dim_B \big( \linf(\Aa^*(\delta))\big)=\delta$. 
In order show  that  $\dim_B(a(x,y)) \ge \delta$ we proceed analogouly to the proof of Theorem \ref{theo:all_arcs_dimB=2}. Let again $n_0$ denote and integer such that the length of the path in the tree $\G(\W_{n_0})$ from $W_{n_0}(x) \in \V(\G(\W_n))$ to $W_{n_0}(y) \in \V(\G(\W_n))$ is at least three. Thus, the arc $a(x,y)$ crosses this square along
 an arc $ a' =  a\cap \big(W \cap \linf(\Aa^*(\delta))\big)$.
The set $\linf\cap W$ is a scaled mixed labyrinth fractal, defined by the sequence of patterns $(\Aa'_k)_{k\ge 1}$, where $\Aa'_k=\Aa^*_{n_0+k}$, for all $k\ge 1$. More precisely,  this new sequence of labyrinth patterns defines the labyrinth fractal $\linf '$, which  is just the image of the dendrite $\linf\cap W$ through a similarity mapping with scaling factor $ m^*_1m^*_2\dots m^*_{n_0} $. Thus, the arc $a'$ is the image through a similarity mapping of an arc between two exits of the fractal $\linf '$. We obtain $\dim_B(a')=\delta$ and since $a'\subset a(x,y)$ it follows that $\dim_B(a(x,y))\ge \dim_B(a')=\delta$, which completes the proof.

\end{proof}

\section{A glimpse into a more general setting}

Let us first recall (see \cite{laby_4x4}) that the \emph{core} of a labyrinth pattern $\Aa$ is the intersection of all labyrinth patterns that are subsets of $\Aa$, in other words a core of a labyrinth pattern is minimal in the sense that colouring any of its white squares in black would lead to a pattern that is not a labyrinth pattern anymore.  It is easy to see that the core of any snake cross pattern is the pattern itself.

\begin{figure}[h!]
\begin{center}
\begin{tikzpicture}[scale=.30]
\draw[line width=1pt] (0,0) rectangle (15,15);
\filldraw[fill=gray!40, draw=black] (7,7) rectangle (8, 8);
\filldraw[fill=orange!40, draw=black] (7,0) rectangle (8, 1);
\filldraw[fill=green!40, draw=black] (7,14) rectangle (8, 15);
\filldraw[fill=blue!40, draw=black] (0,7) rectangle (1, 8);
\filldraw[fill=red!40, draw=black] (14,7) rectangle (15, 8);
%---- jetzt das dicke raster--------
\draw[line width=0.5pt] (1, 0) -- (1,15);
\draw[line width=0.5pt] (2, 0) -- (2,15);
\draw[line width=0.5pt] (3, 0) -- (3,15);
\draw[line width=0.5pt] (4, 0) -- (4,15);
\draw[line width=0.5pt] (5, 0) -- (5,15);
\draw[line width=0.5pt] (6, 0) -- (6,15);
\draw[line width=0.5pt] (7, 0) -- (7,15);
\draw[line width=0.5pt] (8, 0) -- (8,15);
\draw[line width=0.5pt] (9, 0) -- (9,15);
\draw[line width=0.5pt] (10, 0) -- (10,15);
\draw[line width=0.5pt] (11, 0) -- (11,15);
\draw[line width=0.5pt] (12, 0) -- (12,15);
\draw[line width=0.5pt] (13, 0) -- (13,15);
\draw[line width=0.5pt] (14, 0) -- (14,15);

\draw[line width=0.5pt] (0, 1) -- (15,1);
\draw[line width=0.5pt] (0, 2) -- (15,2);
\draw[line width=0.5pt] (0, 3) -- (15,3);
\draw[line width=0.5pt] (0, 4) -- (15,4);
\draw[line width=0.5pt] (0, 5) -- (15,5);
\draw[line width=0.5pt] (0, 6) -- (15,6);
\draw[line width=0.5pt] (0, 7) -- (15,7);
\draw[line width=0.5pt] (0, 8) -- (15,8);
\draw[line width=0.5pt] (0, 9) -- (15,9);
\draw[line width=0.5pt] (0, 10) -- (15,10);
\draw[line width=0.5pt] (0, 11) -- (15,11);
\draw[line width=0.5pt] (0, 12) -- (15,12);
\draw[line width=0.5pt] (0, 13) -- (15,13);
\draw[line width=0.5pt] (0, 14) -- (15,14);

%--- jetzt die schwarzen quadrate-------------
\filldraw[fill=black, draw=black] (0,0) rectangle (1, 1);
\filldraw[fill=black, draw=black] (2,0) rectangle (7, 1);
\filldraw[fill=black, draw=black] (8,0) rectangle (12, 1);
\filldraw[fill=black, draw=black] (13,0) rectangle (15, 1);
\filldraw[fill=black, draw=black] (2,1) rectangle (3, 2);
\filldraw[fill=black, draw=black] (11,1) rectangle (12, 2);
\filldraw[fill=black, draw=black] (13,1) rectangle ( 15, 2);
\filldraw[fill=black, draw=black] (0,2) rectangle (1, 3);
\filldraw[fill=black, draw=black] (2,2) rectangle (4, 3);
\filldraw[fill=black, draw=black] (5,2) rectangle (10, 3);
\filldraw[fill=black, draw=black] (13,2) rectangle (14, 3);
\filldraw[fill=black, draw=black] (0,3) rectangle (1, 4);
\filldraw[fill=black, draw=black] (2,3) rectangle (6, 4);
\filldraw[fill=black, draw=black] (11,3) rectangle (14, 4);
\filldraw[fill=black, draw=black] (0,4) rectangle (1, 5);
\filldraw[fill=black, draw=black] (5,4) rectangle (6, 5);
\filldraw[fill=black, draw=black] (7,4) rectangle (9, 5);
\filldraw[fill=black, draw=black] (10,4) rectangle (13, 5);
\filldraw[fill=black, draw=black] (0,5) rectangle (1, 6);
\filldraw[fill=black, draw=black] (2,5) rectangle (3, 6);
\filldraw[fill=black, draw=black] (4,5) rectangle (6, 6);
\filldraw[fill=black, draw=black] (8,5) rectangle (13, 6);
\filldraw[fill=black, draw=black] (14,5) rectangle (15, 6);
\filldraw[fill=black, draw=black] (0,6) rectangle (3, 7);
\filldraw[fill=black, draw=black] (2,6) rectangle (1, 7);
\filldraw[fill=black, draw=black] (4,6) rectangle (7, 7);
\filldraw[fill=black, draw=black] (8,6) rectangle (9, 7);
\filldraw[fill=black, draw=black] (12,6) rectangle (13, 7);
\filldraw[fill=black, draw=black] (14,6) rectangle (15, 7);
\filldraw[fill=black, draw=black] (2,7) rectangle (3, 8);
\filldraw[fill=black, draw=black] (4,7) rectangle (5, 8);
\filldraw[fill=black, draw=black] (10,7) rectangle (11, 8);
\filldraw[fill=black, draw=black] (12,7) rectangle (13, 8);
\filldraw[fill=black, draw=black] (0,8) rectangle (1, 9);
\filldraw[fill=black, draw=black] (2,8) rectangle (3, 9);
\filldraw[fill=black, draw=black] (6,8) rectangle (7, 9);
\filldraw[fill=black, draw=black] (8,8) rectangle (11, 9);
\filldraw[fill=black, draw=black] (12,8) rectangle (13, 9);
\filldraw[fill=black, draw=black] (14,8) rectangle (15, 9);
\filldraw[fill=black, draw=black] (0,9) rectangle (1, 10);
\filldraw[fill=black, draw=black] (2,9) rectangle (7, 10);
\filldraw[fill=black, draw=black] (9,9) rectangle (11, 10);
\filldraw[fill=black, draw=black] (12,9) rectangle (13, 10);
\filldraw[fill=black, draw=black] (14,9) rectangle (15, 10);
\filldraw[fill=black, draw=black] (0,10) rectangle (1, 11);
\filldraw[fill=black, draw=black] (3,10) rectangle (8, 11);
\filldraw[fill=black, draw=black] (10,10) rectangle (11, 11);
\filldraw[fill=black, draw=black] (14,10) rectangle (15, 11);
\filldraw[fill=black, draw=black] (0,11) rectangle (4, 12);
\filldraw[fill=black, draw=black] (9,11) rectangle (12, 12);
\filldraw[fill=black, draw=black] (13,11) rectangle (15, 12);
\filldraw[fill=black, draw=black] (0,12) rectangle (2, 13);
\filldraw[fill=black, draw=black] (5,12) rectangle (11, 13);
\filldraw[fill=black, draw=black] (14,12) rectangle (15, 13);
\filldraw[fill=black, draw=black] (3,13) rectangle (4, 14);
\filldraw[fill=black, draw=black] (10,13) rectangle (11, 14);
\filldraw[fill=black, draw=black] (12,13) rectangle (13, 14);
\filldraw[fill=black, draw=black] (14,13) rectangle (15, 14);
\filldraw[fill=black, draw=black] (1,14) rectangle (2, 15);
\filldraw[fill=black, draw=black] (3,14) rectangle (7, 15);
\filldraw[fill=black, draw=black] (8,14) rectangle (9, 15);
\filldraw[fill=black, draw=black] (10,14) rectangle (11, 15);
\filldraw[fill=black, draw=black] (12,14) rectangle (13, 15);
\filldraw[fill=black, draw=black] (14,14) rectangle (15, 15);

\end{tikzpicture}
\caption{Example: a labyrinth pattern of width $15$ whose core is the snake cross pattern $\Aa_2$ shown in Figures \ref{fig:A1A2A3} and \ref{fig:colorA2}. The central square and the pattern's exits are marked in colour}
\label{fig:ex:special_cross}
\end{center}
\end{figure}
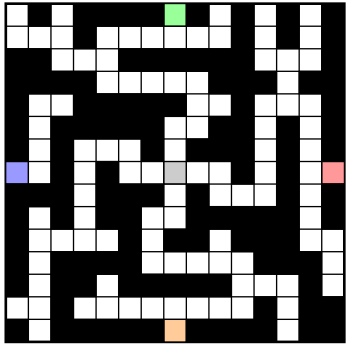

Let us denote by $\Aa_k(2)$, for $k\ge 1$, the $k$-th snake cross pattern that occurs in the construction used in the proof of Theorem \ref{theo:all_arcs_dimB=2} and $\linf(2)$ the corresponding mixed labyrinth fractal.

\begin{corollary}\label{cor:cores_same_dim}
 If ${\Aa_k}$ is a sequence of labyrinth patterns such that the core of $\Aa_k$ is  $\Aa_k^*(2)$, then  the box counting dimension of the resulting fractal $\linf$ is $\dim_B(\linf)=2$  and every arc that connects distinct points  in this dendrite has box counting dimension $2$.
\end{corollary}

\begin{proof}
The fact that $\dim_B(\linf(2))\le 2$ is basic, since the box counting dimension of the unit square is 2. Forthermore, $\linf(2)\subset \linf$ implies $\dim_B(\linf)\ge 2$. 
For $x,y \in \linf$, $x\ne y$, using the same ideas as in the proof of Theorem \ref{theo:all_arcs_dimB=2} we obtain that $\dim_B(a(x,y))\ge 2$, which, together to the converse inequality, implies $\dim_B(a(x,y))=2$.

\end{proof}

We conclude with a more general statement.

\begin{conjecture}\label{cor:cores_same_dim}
Let $\delta \in [1,2]$ If  $\{\Aa_k^*\}_{k\ge 1}$ is a sequence of labyrinth patterns such that the core of $\Aa_k$ is  $\Aa_k^*(\delta )$,  i.e., the $k$-th snake cross pattern used in the construction of $\linf (\Aa^*(\delta))$. 
Then every arc that connects distinct points  in this dendrite has box counting dimension $\delta$.
\end{conjecture}


\begin{thebibliography}{10}

\bibitem{AnhHoffmanSeegerTarafdar2005}
D.~H.~N. Anh, K.~H. Hoffmann, S.~Seeger, and S.~Tarafdar.
\newblock Diffusion in disordered fractals.
\newblock {\em EPL (Europhysics Letters)}, 70(1):109, 2005.



\bibitem{laby_4x4}
L.~L. {Cristea} and B.~{Steinsky}.
\newblock {Curves of infinite length in $4 \times 4$-labyrinth fractals.}
\newblock {\em {Geom. Dedicata}}, 141:1--17, 2009.

\bibitem{laby_mxm}
L.~L. {Cristea} and B.~{Steinsky}.
\newblock {Curves of infinite length in labyrinth fractals.}
\newblock {\em {Proc. Edinb. Math. Soc., II. Ser.}}, 54(2):329--344, 2011.

\bibitem{mixlaby}
L.~L. {Cristea} and B.~{Steinsky}.
\newblock {Mixed labyrinth fractals.}
\newblock {\em {Topology Appl.}}, 229:112--125, 2017.

\bibitem{arcs_mixlaby}
L.~L. {Cristea} and G.~{Leobacher}.
\newblock {On the length of arcs in labyrinth fractals.}
\newblock {\em {Monatsh. Math.}}, 185(4):575--590, 2018.

\bibitem{supermix}
L.~L. {Cristea} and G.~{Leobacher}.
\newblock {Supermixed labyrinth fractals.}
\newblock {\em {J. Fractal Geom.}}, to appear.


%\bibitem{Edgar}
%G.~{Edgar}.
%\newblock {\em {Measure, topology, and fractal geometry. 2nd ed.}}
%\newblock New York, NY: Springer, 2nd ed. edition, 2008.


\bibitem{falconer_book}
K.~{Falconer}.
\newblock {\em {Fractal geometry. Mathematical foundations and applications.
  3rd ed.}}
\newblock Hoboken, NJ: John Wiley \& Sons, 3rd ed. edition, 2014.

\bibitem{FengWenWu_dim_homogeneous_Moran_sets1997}
D.~Feng, Z.~Wen, J.~Wu.
\newblock Some dimensional results for homogeneous Moran sets.
\newblock {\em Science in China (Series A)}, 40(5):475--482, May 1997.


\bibitem{tarafdar_multifractalNaCl2013}
A.~Giri, M.~Dutta~Choudhury, T.~Dutta, and S.~Tarafdar.
\newblock Multifractal growth of crystalline NaCl aggregates in a gelatin
  medium.
\newblock {\em Crystal Growth \& Design}, 13(1):341--345, Jan 2013.

\bibitem{GrachevPotapovGerman2013}
A.~A. Potapov, V.~A. German, and V.~I. Grachev.
\newblock Fractal labyrinths as a basis for reconstruction planar
  nanostructures.
\newblock In {\em 2013 International Conference on Electromagnetics in Advanced
  Applications (ICEAA)}, pages 949--952, Sept 2013.


\bibitem{JanaGarcia_lithiumdendrite2017}
A.~{Jana} and R.~E. {Garc\'{i}a}.
\newblock Lithium dendrite growth mechanisms in liquid electrolytes.
\newblock {\em Nano Energy}, 41:552 -- 565, 2017.

\bibitem{Kuratowski}
K.~Kuratowski.
\newblock {\em {Topology, Volume II}}.
\newblock Academic Press, 1968.

%\bibitem{Peano} G. Peano, \emph{Sur une courbe, qui remplit une aire plane}, Math. Ann. \emph{36}, 157-160, 1890

\bibitem{PotapovPotapovPotapov_dec2017}
A.~Potapov and V.~Potapov.
\newblock Fractal radioelement's, devices and systems for radar and future
  telecommunications: Antennas, capacitor, memristor, smart 2d
  frequency-selective surfaces, labyrinths and other fractal metamaterials.
\newblock {\em Journal of International Scientific Publications: Materials,
  Methods \& Technologies}, 11:492--512, 2017.


\bibitem{PotapovGermanGrachev2013}
A.~A. Potapov, V.~A. German, and V.~I. Grachev.
\newblock ``nano -'' and radar signal processing: Fractal reconstruction
  complicated images, signals and radar backgrounds based on fractal
  labyrinths.
\newblock In {\em 2013 14th International Radar Symposium (IRS)}, volume~2,
  pages 941--946, June 2013.

\bibitem{PotapovZhang2016}
A.~A. Potapov and W.~Zhang.
\newblock Simulation of new ultra-wide band fractal antennas based on fractal
  labyrinths.
\newblock In {\em 2016 CIE International Conference on Radar (RADAR)}, pages
  1--5, Oct 2016.

\bibitem{SamuelTetenov_selfsimilar_dendrites}
M.~{Samuel}, A.~V. {Tetenov}, and D.~A. {Vaulin}.
\newblock {Self-similar dendrites generated by polygonal systems in the plane.}
\newblock {\em {Sib. \`Elektron. Mat. Izv.}}, 14:737--751, 2017.

\bibitem{SeegerHoffmannEssex2009_randomKoch}
S.~{Seeger}, K.~H. {Hoffmann}, and C.~{Essex}.
\newblock {Random walks on random Koch curves.}
\newblock {\em {J. Phys. A, Math. Theor.}}, 42(22):11, 2009.


\bibitem{Tricot_book}
C.~{Tricot}.
\newblock {\em {Curves and fractal dimension. With a foreword by Michel
  Mend\`es France. Transl. from the French.}}
\newblock New York, NY: Springer-Verlag, 1995.


\end{thebibliography}
\end{document}